\documentclass[11 pt, reqno]{amsart}
\usepackage{bbm}
\usepackage{graphics}
 \usepackage{setspace}
\usepackage{url}
\usepackage[utf8]{inputenc}
\usepackage{chngcntr}
\usepackage{lipsum}
\usepackage{amsthm}
\usepackage{amsfonts}
	\usepackage{cases}
\usepackage{color}
\usepackage{geometry}
\usepackage{amsmath}
\usepackage{amssymb}
\usepackage{latexsym}
\usepackage{amssymb}
\usepackage{mathrsfs}
\usepackage{amsmath}
\usepackage{enumerate}
\usepackage{hyperref}
\hypersetup{
    colorlinks=true,                          
    linkcolor=blue, 
    citecolor=red, 
    urlcolor=blue  } 

\usepackage{eurosym}
	\usepackage{cases}
\usepackage{color}

\newcommand{\kom}[1]{}
\renewcommand{\kom}[1]{{\bf [#1]}}

 \def\1{\raisebox{2pt}{\rm{$\chi$}}}

\newtheorem{theorem}{Theorem}[section]
\newtheorem{corollary}[theorem]{Corollary}
\newtheorem{lemma}[theorem]{Lemma}

\newtheorem{definition}[theorem]{Definition}

\newtheorem{remark}[theorem]{Remark}

\numberwithin{equation}{section}

\newcommand{\R}{{\mathbb R}}

\newcommand{\N}{{\mathbb N}}

 \newcommand{\eps}{{\varepsilon}}
 \def\1{\raisebox{2pt}{\rm{$\chi$}}}
 
\newcommand{\abs}[1]{\left|#1\right|}
\newcommand{\norm}[1]{\left|\left|#1\right|\right|}

\newcommand{\osc}{\operatorname{osc}}

\def\vint_#1{\mathchoice%
          {\mathop{\kern 0.2em\vrule width 0.6em height 0.69678ex depth -0.58065ex
                  \kern -0.8em \intop}\nolimits_{\kern -0.4em#1}}%
          {\mathop{\kern 0.1em\vrule width 0.5em height 0.69678ex depth -0.60387ex
                  \kern -0.6em \intop}\nolimits_{#1}}%
          {\mathop{\kern 0.1em\vrule width 0.5em height 0.69678ex
              depth -0.60387ex
                  \kern -0.6em \intop}\nolimits_{#1}}%
          {\mathop{\kern 0.1em\vrule width 0.5em height 0.69678ex depth -0.60387ex
                  \kern -0.6em \intop}\nolimits_{#1}}}
\def\vintslides_#1{\mathchoice%
          {\mathop{\kern 0.1em\vrule width 0.5em height 0.697ex depth -0.581ex
                  \kern -0.6em \intop}\nolimits_{\kern -0.4em#1}}%
          {\mathop{\kern 0.1em\vrule width 0.3em height 0.697ex depth -0.604ex
                  \kern -0.4em \intop}\nolimits_{#1}}%
          {\mathop{\kern 0.1em\vrule width 0.3em height 0.697ex depth -0.604ex
                  \kern -0.4em \intop}\nolimits_{#1}}%
          {\mathop{\kern 0.1em\vrule width 0.3em height 0.697ex depth -0.604ex
                  \kern -0.4em \intop}\nolimits_{#1}}}

\newcommand{\aveint}[2]{\mathchoice%
          {\mathop{\kern 0.2em\vrule width 0.6em height 0.69678ex depth -0.58065ex
                  \kern -0.8em \intop}\nolimits_{\kern -0.45em#1}^{#2}}%
          {\mathop{\kern 0.1em\vrule width 0.5em height 0.69678ex depth -0.60387ex
                  \kern -0.6em \intop}\nolimits_{#1}^{#2}}%
          {\mathop{\kern 0.1em\vrule width 0.5em height 0.69678ex depth -0.60387ex
                  \kern -0.6em \intop}\nolimits_{#1}^{#2}}%
          {\mathop{\kern 0.1em\vrule width 0.5em height 0.69678ex depth -0.60387ex
                  \kern -0.6em \intop}\nolimits_{#1}^{#2}}}

\newcommand{\ol}{\overline}

\newcommand{\dist}{\operatorname{dist}}

\newcommand{\vp}{\varphi}

\newcommand{\tr}{\operatorname{tr}}

\newcommand{\A}{\mathcal A}
\begin{document}
\author[A. Attouchi]{Amal Attouchi}
\address{Department of Mathematics and Statistics, University of Jyv\"askyl\"a, PO~Box~35, FI-40014 Jyv\"askyl\"a, Finland}
\email{amal.a.attouchi@jyu.fi}
\email{eero.k.ruosteenoja@jyu.fi}

\author[Ruosteenoja]{Eero Ruosteenoja}
\title{Gradient regularity for a singular parabolic equation in non-divergence form}
\date{\today}
\begin{abstract}
In this paper we consider viscosity solutions of a class of non-homogeneous singular parabolic
equations $$\partial_t u-|Du|^\gamma\Delta_p^N u=f,$$ where $-1<\gamma<0$, $1<p<\infty$, and $f$ is a given bounded function. We establish interior H\"older regularity of the gradient by studying two alternatives: The first alternative uses an iteration which is based on an approximation lemma. In the second alternative we use a small perturbation argument.
\end{abstract}
\maketitle
\section{Introduction}
\label{sec:intro}
We study gradient regularity of the following  singular parabolic equation in non-divergence form,
\begin{equation}\label{maineq}
	\partial_t u-|Du|^\gamma\Delta_p^N u=f\quad \quad \text{in}\quad Q_1:=B_1(0)\times (-1,0)\subset \R^n\times \R.
\end{equation}
Here $$\Delta_p^Nu:=\Delta u+(p-2) \left\langle D^2u\frac{Du}{\abs{Du}}, \frac{Du}{\abs{Du}}\right\rangle$$ is the normalized $p$-Laplacian. We assume that $-1<\gamma<0$, $1<p<\infty$, and $f$ is a given continuous and bounded function. 

When $\gamma=p-2$, equation \eqref{maineq} is the standard parabolic $p$-Laplace equation $$u_t-\Delta_p u=f,$$ and in that case it is possible to consider both  distributional weak solutions and viscosity solutions. In the case of bounded weak solutions, equivalence with viscosity solutions was shown by Juutinen, Lindqvist, and Manfredi \cite{julm}. For that equation, H\"older regularity of the gradient was shown by DiBenedetto and Friedman \cite{df85} and Wiegner \cite{wie86}, see also Kuusi and Mingione \cite{KM12} and references therein.

Another special case is $\gamma=0$, when the equation reads $$u_t-\Delta^N_p u=f.$$ The motivation to study parabolic equations involving the normalized $p$-Laplacian stems partially from connections to time-dependent tug-of-war games \cite{MPR10,parviainenr16,Han18} and image processing \cite{does11}. For regularity results concerning this equation, we refer to \cite{Banerjeeg15,AP17,Bertim19,Hoegl19,Dong19}.

Demengel \cite{dem11} proved existence, uniqueness and H\"older regularity for solutions of a class of singular or degenerate parabolic equations including \eqref{maineq}, see also \cite{bata, OS} and references therein. Argiolas, Charro, and Peral \cite{arg} showed Aleksandrov-Bakelman-Pucci type estimate, whereas Parviainen and V\'azquez \cite{PV18} showed parabolic Harnack's inequality for radial solutions. In the case $\gamma=0$ and $f\equiv 0$, Jin and Silvestre \cite{jinsl15}  showed the H\"older regularity of the gradient for solutions of \eqref{maineq}, and the result was generalized by Imbert, Jin, and Silvestre \cite{IJS17} to the whole range $-1<\gamma<\infty$. In the non-homogeneous case, Attouchi and Parviainen \cite{AP17} treated $C^{1,\alpha}$-regularity in the uniformly parabolic case $\gamma=0$, and later the same result was proved by Attouchi \cite{att19} in the degenerate case $\gamma\in (0, \infty)$. For related regularity results in the elliptic setting, we refer to \cite{wang94,Birindellid10,IS13,Birindellid14,Attouchipr17,Attouchir18,BM19}.

In this paper we continue the study of $C^{1,\alpha}$-regularity by focusing on the range $\gamma\in (-1,0)$. Our main result is the following.

\begin{theorem}\label{paalause}
Let $u$ be a viscosity solution of equation \eqref{maineq}, where $-1<\gamma<0$, $1<p<\infty$, and $f$ is a continuous and bounded function. There exist $\alpha=\alpha(p,n, \gamma)\in(0,1)$ and $C=C(p,n, \gamma,\norm{u}_{L^\infty(Q_1)},\norm{f}_{L^\infty(Q_1)})>0$ such that when $(x,t),(y,s)\in Q_{1/2}$,
$$|Du(x,t)-Du(y,s)|\leq C(|x-y|^\alpha+|t-s|^{\beta})$$
and 
$$|u(x,t)-u(x,s)|\leq C|t-s|^\sigma,$$
where $\beta:= \frac{\alpha}{2-\alpha \gamma}$ and $\sigma:=\frac{1+\alpha}{2-\alpha \gamma}$.
\end{theorem}
Our proof relies on the method of alternatives and the improvement of flatness. The strategy is to define a process that provides a better linear approximation in a smaller cylinder, and which we can iterate until we reach a cylinder where a so called smooth alternative holds. More precisely, we define an induction hypothesis based on the size of the slope, see Corollary \ref{iterimp}. In order to proceed with the iteration, we use an intrinsic scaling together with the approximation lemma \cite[Lemma 4.1]{att19}. This lemma enables us to consider the solution of \eqref{maineq} as a perturbation of the solution of the corresponding homogeneous equation when the absolute size of $f$ is sufficiently small. We may assume this by scaling.  

The induction step can only work indefinitely if the gradient vanishes. In the case the iteration stops, our strategy is to show that the solution is sufficiently close to a linear function with a non-vanishing gradient. We show that in that case the solution itself has a gradient that is bounded away from zero in some cylinder. Hence, the equation is uniformly parabolic and no longer singular, so we can use the general regularity result from \cite{LU68,Lieb96}. 

We remark that this method is flexible enough to be applied to the study of the gradient regularity for solutions of a more general class of singular, fully nonlinear parabolic equations of the type 
\[
u_t-|Du|^\gamma F(D^2 u)=f,
\] 
or those considered in \cite{dem11}, once the regularity for the corresponding homogeneous case has been treated.

This paper is organised as follows. In Section \ref{prel} we fix the notation and provide a Lipschitz estimate (Lemma \ref{toulouse}), which is used in the non-singular alternative. In Section \ref{applemmas} we prove lemmas related to the improvement of flatness and iteration, and in Section \ref{sect5} we prove Theorem \ref{paalause}. The technical proof of Lemma \ref{toulouse} is postponed to Section \ref{ishii}. \\

\noindent \textbf{Acknowledgement.} A.A. was supported by the Academy of Finland, project number 307870. 

\section{Preliminaries}\label{prel}
\subsection{Notation}
We denote parabolic cylinders by $$Q_r(x,t):=B_r(x)\times (t-r^2,t), $$ where $B_r(x)\subset \R^n$ is a ball centered at $x$ with radius $r>0$. We will also use intrinsic parabolic cylinders $$Q^a_r(x,t):=B_r(x)\times (t-r^2a^{-\gamma},t), $$ where $\gamma\in (-1,0)$. We denote $$Q_r:=Q_r(0,0)\quad \text{and}\quad Q^{a}_r:=Q^{a}_r(0,0). $$

We define viscosity (super-, sub-) solutions of equation \eqref{maineq} as follows.

\begin{definition}
A locally bounded and lower semi-continuous function $u$ is a viscosity supersolution of \eqref{maineq}, if at any point $(x_0,t_0)\in Q_1$, one of the following holds: 
\begin{enumerate}
\item For any $\phi\in C^{2,1}(Q_1)$, $D\phi(x_0,t_0)\neq 0$, touching $u$ from below at $(x_0,t_0)$, it holds 
\[
\partial_t \phi(x_0,t_0)-|D\phi(x_0,t_0)|^\gamma\Delta_p^N \phi\geq f(x_0,t_0).
\]
\item If there is $\delta>0$ and $\phi\in C^1(t_0-\delta,t_0+\delta)$ satisfying $\phi(t_0)=0$, $u(x_0,t_0)\leq u(x_0,t)-\phi(t)$ for any $t\in (t_0-\delta,t_0+\delta)$, and $\inf_{(t_0-\delta,t_0+\delta)}(u(x,t)-\phi(t))$ is a constant in some neighborhood of $x_0$, then it holds 
\[
\phi'(t_0)\geq f(x_0,t_0).
\]
\end{enumerate}

Similarly, viscosity subsolutions are defined changing touching from below by touching from above, inf by sup, and $\geq$ by $\leq$.

A continuous function is a viscosity solution of \eqref{maineq} if it is both viscosity sub- and supersolution. 
\end{definition}

Without a loss of generality, we assume that $u(0,0)=0$, $\norm{u}_{L^\infty(Q_1)}\leq \frac12$ and $\norm{f}_{L^\infty(Q_1)}\leq \eps_0,$ where $\eps_0>0$ will be fixed in Lemma \ref{approl} below. Indeed, we can use the scaling $$u_\theta(x,t)=\theta u(\theta^{\frac{-\gamma}{2+\gamma}}x,t),$$ where $$\theta=\frac{1}{2\norm{u}_{L^\infty(Q_1)}+\frac{\norm{f}_{L^\infty(Q_1)}}{\eps_0}+1}.$$

\subsection{Intermediate lemmas}
In this section we  gather some intermediate lemmas that will play a role in the proof of the Hölder regularity of  the spatial gradient. First we recall the following result on the Lipschitz estimates for solutions to \eqref{maineq}.
\begin{lemma}\label{gort} \emph{(}\cite[Lemma 3.1]{att19}\emph{)}
Let $-1<\gamma<\infty$, $1<p<\infty$ and $f\in C(Q_1)\cap L^\infty(Q_1)$.		Let $u$ be a bounded viscosity solution to equation \eqref{maineq}. There exists a constant $C=C(p,n,\gamma)>0$ such that   for all  $(x,t),(y,t)\in Q_{7/8}$, it holds
		\begin{equation*}
		\begin{split}
		\abs{ u(x,t)-u(y,t)}\leq C \left(\norm{  u}_{L^\infty(Q_1)}+\norm{  u}_{L^\infty(Q_1)}^{\frac{1}{1+\gamma}}+\norm{f}_{L^\infty(Q_1)}^{\frac{1}{1+\gamma}}\right) \abs{x-y}.
		\end{split}
		\end{equation*}
	\end{lemma}
	
Next, we consider bounded  solutions of
\begin{equation}\label{deviaatio}
\partial_t w-|Dw+K|^\gamma\left[\Delta w+(p-2) \dfrac{\langle D^2 w (Dw +K), (\ Dw+K)\rangle}{|Dw+K|^2}\right]=\tilde f\quad\text{in}\quad Q_1,
\end{equation}
with $w(0)=0$, $\norm{w}_{L^\infty(Q_1)}\leq 1, \norm{\tilde{f}}_{L^\infty(Q_1)}\leq 1$ and $|K|\geq 1$. 
The previous  Lemma provides a first Lipschitz estimate for $w$. Indeed, since $w(x,t)+K\cdot x$ is a solution to \eqref{maineq}, $1\leq |K|$ and $\norm{w}_{L^\infty(Q_1)}\leq 1$, we get
\begin{equation}\label{levicopar}
		\begin{split}
		\abs{ w(x,t)-w(y,t)}\leq C \left(\norm{  w}_{L^\infty(Q_1)}+|K|^{\frac{1}{1+\gamma}}+\norm{\tilde f}_{L^\infty(Q_1)}^{\frac{1}{1+\gamma}}\right) \abs{x-y}\leq \bar C|K|^{\frac{1}{1+\gamma}}|x-y|,
		\end{split}
		\end{equation}
		for some $\bar C=\bar C(p,n,\gamma)$.
However, we can improve this estimate and provide a better control on the gradient.  This estimate will play a key role in the non-singular alternative.		
\begin{lemma}\label{toulouse} Let $-1<\gamma<0$ and $1<p<\infty$.
Let $K\in \R^n$ with $2\leq |K|\leq M$ for some $M=M(p,n, \gamma)>0$.
Let $w$ be a viscosity solution of \eqref{deviaatio}, with $w(0,0)=0$. There exists $\eta_1=\eta_1(p,n,\gamma)$ such that if $\norm{w}_{L^\infty(Q_1)}\leq \eta_1$ and $\norm{\tilde f}_{L^\infty(Q_1)}\leq \eta_1$, then $w$ is locally  Lipschitz continuous in space and for $(x,t), (y,t)\in Q_{3/4}$ we have
$$|w(x,t)-w(y,t)|\leq |x-y|.$$
\end{lemma}
The proof  makes use of the Ishii-Lions method. It proceeds in two steps: first we obtain good enough H\"older estimates, and then use the H\"older estimate and the Ishii-Lions method again to prove the desired Lipschitz estimates. We postpone the technical proof to Section \ref{ishii} in order to keep the paper easier to read.

\section{Approximation Lemmas and iteration}\label{applemmas}
In this section we state the approximation lemmas needed  to implement the iteration, and define the induction hypothesis. Let us first recall an approximation result from \cite{att19}.
	
	\begin{lemma}\label{approl}\emph{(}\cite[Lemma 4.1]{att19}\emph{)}
Let $-1<\gamma<0$ and $1<p<\infty$. Let $u$ be a viscosity solution to \eqref{maineq} with $\underset{Q_1}{\osc} \, u\leq 1$.	For every $\tau >0$, there exists  $\eps_0(p,n, \gamma, \tau)\in (0, 1)$ such that  if $\norm{f}_{L^\infty(Q_1)}\leq \eps_0$, then there exists a solution $\tilde u$ to 
	\begin{equation}\label{homeq1}\partial_t\tilde u=|D\tilde u|^\gamma\Delta_p^N\tilde u\quad\text{in}\quad Q_{11/16}
	\end{equation}
	such that 
	$$\norm{u-\tilde u}_{L^\infty(Q_{5/8})}\leq \tau.$$	
	\end{lemma}

Next we define a first linear approximation lemma by combining the previous lemma with the regularity result of \cite[Theorem 1.1]{IJS17}. Since uniform Lipschitz estimates for deviations from planes don't seem to be available in the singular case, we refine \cite[Lemma 4.2]{att19} by adding the parameter $\eta_1$. This allows us to handle the case where the iteration stops by making $u$ as close as needed to a linear function with a non-vanishing gradient. For the following lemma, recall that
$$Q_{\rho}^{(1-\delta)}=B_{\rho}\times (-\rho^{2}(1-\delta)^{-\gamma}, 0].$$

	\begin{lemma}\label{starting}
	Let $-1<\gamma<0$ and $1<p<\infty$. 
	 Let $\eta_1$ be the constant defined in Lemma \ref{toulouse}. Let $u$ be a viscosity solution to \eqref{maineq} such that $\underset{Q_1}{\osc}\, u\leq 5$. There exist $\eps_0=\eps_0(p,n, \gamma)>0$, such that if $\norm{f}_{L^\infty(Q_1)} \leq \eps_0$, then there exist  $\rho=\rho(p,n, \gamma)>0$  and $\delta=\delta(p,n, \gamma)\in (0,1/2)$ with $\rho<(1-\delta)^{\gamma+1}$ and  a vector $q$ with $|q|\leq B=B(n, p, \gamma)$ such that
	$$\underset{(x,t)\in Q_{\rho}^{(1-\delta)}}{\osc}\, (u(x,t)-q\cdot x)\leq \eta_1\rho(1-\delta).$$

	\end{lemma}
	\begin{proof}

	Let $\tilde u$ be the viscosity solution to 
	$$\partial_t \tilde u-|D \tilde u|^\gamma\Delta_p^N \tilde u=0\qquad\text{in}\qquad Q_{11/16},$$
coming from Lemma \ref{approl}.
	From the  regularity result of  \cite[Theorem 1.1]{IJS17}, there exists $C_0=C_0(p,n, \gamma)>0$ and $\beta_0=\beta_0(p,n,\gamma)$ such that for all $\mu\in (0, 5/8)$ there exists
 $q$ with $|q|\leq B=B(p,n, \gamma)$ such that

	$$\underset{(x,t)\in Q_{\mu}}{\osc}\, (\tilde u(x,t)-q\cdot x)\leq C_0(1+\norm{\tilde u}_{L^\infty (Q_{5/8})})\mu^{1+ \beta_0}.$$
	It is important to notice that $B$ depends only on $p,n, \gamma$.
	We choose $\mu_0\in (0, 5/8)$ such that 
$$\underset{(x,t)\in Q_{\mu_0}}{\osc}\, (\tilde u(x,t)-q\cdot x)	\leq \frac{\eta_1}{4}\mu_0(1-\delta)^{\gamma+2}$$
for some $\delta\in (0,1/2)$.
Thus there exist two constants $\rho$ and $\delta$ depending on $p, n, \gamma$ such that
	$$\underset{(x,t)\in Q_{\rho}^{1-\delta}}{\osc}\, (\tilde u(x,t)-q\cdot x)\leq \frac{\eta_1}{4} \rho (1-\delta),$$
	with $\rho=\mu_0(1-\delta)^{\gamma+1}<(1-\delta)^{\gamma+1}$. It follows from Lemma \ref{approl} that for $\tau:=\dfrac{\eta_1}{4} \rho(1-\delta)$ there exists $\eps_0$ such that if  $\norm{f}_{L^\infty(Q_1)}\leq \eps_0$,  we have
	\begin{align*}
	\underset{(x,t)\in Q_{\rho}^{1-\delta}}{\osc}\, (u(x,t)-q\cdot x)
	&\leq \underset{(x, t)\in Q_{\mu_0}}{\osc}\, (u(x,t)-\tilde u(x,t))+\underset{(x,t)\in Q_{\rho}^{1-\delta}}{\osc}\, (\tilde u(x,t)-q\cdot x)\\
	&\leq \tau+\frac{\eta_1}{4} \rho(1-\delta)\\
	&
	\leq \eta_1\rho(1-\delta).
	\end{align*}
	The choice of $\tau$ determines the smallness of $f$.
	\end{proof}

From now on we may assume that $\eps_0<\eta_1$. Next we treat the situation of vanishing slope. Now with $\rho, \delta,\eps_0$ as in lemma \ref{starting} and $\eta_1$  as in lemma \ref{toulouse}, we have the following iteration.

\begin{corollary}\label{iterimp}
Let $u$ be a viscosity solution to \eqref{maineq} such that $\osc_ {Q_1} u\leq 1$.
Let $\eta_1$ be the constant coming from Lemma \ref{toulouse}, and $B$ and $\eps_0$ the constants coming from Lemma \ref{starting}. Assume that $\norm{f}_{L^\infty(Q_1)} \leq \eps_0$. Then there exist  $\rho=\rho(p,n, \gamma)>0$  and $\delta=\delta(p,n, \gamma)\in (0,1)$ with $\rho<(1-\delta)^{\gamma+1}$ such that, for every nonnegative integer $k$, if it holds
 \begin{equation}\label{conditionimpro}
\quad\left\{\begin{array}{ll}
\exists\, \, l_i \,\text{with}\,\,|l_i|\leq 2(1-\delta)^i\quad \text{such that}\,\underset{(x,t)\in Q_{\rho^i}^{(1-\delta)^i}}{\osc}\, (u(x,t)-l_i\cdot x)\leq\eta_1 \rho^i(1-\delta)^i\\
 \text{for}\quad i=0,\ldots , k,
\end{array}\right.
\end{equation}
then there exists a vector $l_{k+1}$ such that 
$$\underset{(x,t)\in Q_{\rho^{k+1}}^{(1-\delta)^{k+1}}}{\osc}\, (u(x,t)-l_{k+1}\cdot x)\leq \eta_1\rho^{k+1}(1-\delta)^{k+1},$$
and $$|l_{k+1}|\leq B(1-\delta)^k,\quad \quad |l_{k+1}-l_k|\leq C_1(1-\delta)^k,$$
with $C_1=C_1(p,n, \gamma)>0$.
\end{corollary}
\begin{proof}
We set $C_1=B+2$. 
For $j=0$ we take  $l_0=0$, and  the result follows from Lemma \ref{starting},  since  $\osc_{Q_1} u\leq 1$. Suppose that the result of the Lemma \ref{starting} holds for $j= 0, \ldots , k $. We are going to  prove it for $ j= k+1$. Define
$$w_k(x,t):=\frac{u(\rho^k x, \rho^{2k}(1-\delta)^{-k\gamma} t)-l_k\cdot \rho^k x}{\rho^k(1-\delta)^k},$$
and
\[
\bar f(x,t):=\rho^k(1-\delta)^{-k(\gamma+1)} f(\rho^k x, \rho^{2k}(1-\delta)^{-k\gamma} t).
\]
By assumption,  we have $\osc_{Q_1} w_k\leq \eta_1\leq 1$ and $|l_k|\leq 2 (1-\delta)^k$. 
Let $h=\frac{l_k}{(1-\delta)^k}$. The function $\bar v(x,t):= w_k(x,t)+h\cdot x$   satisfies 
$$\osc_{Q_1} \bar v\leq 1+2|h|\leq 5,$$
and solves in $Q_1$
$$\partial_t\bar v=|D\bar v|^\gamma\Delta_p^N\bar v+\bar f.$$
Due to $\rho(1-\delta)^{-(\gamma+1)}<1$, we also have
$$\norm{\bar f}_{L^\infty(Q_1)}\leq \eps_0<1.$$
Then Lemma \ref{starting} implies the existence of $q$ with $|q|\leq B=B(p,n, \gamma)$ such that 
$$\underset{(x,t)\in Q_{\rho}^{(1-\delta)}}{\osc}\, (\bar v(x,t)-q\cdot x)\leq \eta_1\rho(1-\delta).$$
Going back to $u$, we have
	$$\underset{(x,t)\in Q_{\rho}^{(1-\delta)}}{\osc}\, (u(\rho^k x, \rho^{2k}(1-\delta)^{-k\gamma} t)-\rho^k(1-\delta)^k q \cdot x)\leq \eta_1\rho^{k+1}(1-\delta)^{k+1}.$$
Scaling back, we conclude
\begin{equation*}
\underset{(x,t)\in Q_{\rho^{k+1}}^{(1-\delta)^{k+1}}}{\osc}\, (u(x,t)-l_{k+1}\cdot x)\leq \eta_1\rho^{k+1}(1-\delta)^{k+1},
\end{equation*}
where 
\begin{equation*}
l_{k+1}:= (1-\delta)^k q
\end{equation*}
satisfies $|l_{k+1}-l_k|\leq( B+2)(1-\delta)^k:= C_1(1-\delta)^k$.
\end{proof}

	\section{Proof of the Hölder regularity of the spatial derivatives}\label{sect5}
We are now in a position to prove the H\"older continuity of  $Du$  at the
origin and   the improved Hölder regularity of $u$ with respect to the time variable.

\begin{lemma}

Let $-1<\gamma<0$, $1<p<\infty$, and let $u$ be a viscosity solution to \eqref{maineq} with $\underset{Q_1}{\osc}\, u\leq 1$. Let $\eps_0$ be the constant coming from Lemma \ref{starting}, and assume that $\norm{f}_{L^\infty(Q_1)}\leq \eps_0$. Then there exist $\alpha=\alpha(p,n, \gamma)\in (0,1)$ and $C=C(p,n, \gamma)>0$ such that
$$|Du(x,t)-Du(y,s)|\leq C(|x-y|^\alpha+|t-s|^{\beta})$$
and 
$$|u(x,t)-u(x,s)|\leq C|t-s|^\sigma,$$
where $\beta:=\frac{\alpha}{2-\alpha \gamma}$ and $\sigma:=\frac{1+\alpha}{2-\alpha \gamma}$.

\end{lemma}
\begin{proof}
Let $\rho$ and $\delta$  be the constants coming from Lemma \ref{starting}.
Let $k$ be  the minimum nonnegative integer such that
the condition \eqref{conditionimpro} does not hold.
We can conclude from Lemma \ref{iterimp} that for any vector  $\xi$ with $|\xi|\leq 2(1-\delta)^k$, it holds 
\begin{equation}
|u(t,x)-\xi\cdot x|\leq C(|x|^{1+\tau}+|t|^{\frac{1+\tau}{2-\tau\gamma}})\quad \text{for}\quad (x,t)\in Q_1\setminus Q_{\rho^{k+1}}^{(1-\delta)^{k+1}},
\end{equation}
	where $\tau:=\frac{\log(1-\delta)}{\log(\rho)}$
and $C=\frac{3+ C_1(1-\delta)^{-1}}{\rho(1-\delta)}$.
Next we treat differently the following two cases.\\

\noindent {\bf First case:} $k=\infty$. The regularity result holds with $$\alpha=\min(1,\tau)=\min\left(1,\dfrac{\log(1-\delta)}{\log \rho}\right)\in \left(0,1\right).$$
Indeed, for all $k\in\N$, there exists $l_k\in \R^n$  with $|l_k|\leq 2 (1-\delta)^k$ such that
		\begin{equation*}
		\underset{(y, t)\in Q^{(1-\delta)^k}_{\rho^{k}}}{\osc} \, (u(y,t)-l_{k}\cdot y)\leq \eta_1 \rho^k(1-\delta)^k.
		\end{equation*}
We conclude the result by using the characterization of functions with H\"older continuous gradient, see also \cite{Lieb96,AP17}.\\
	
\noindent {\bf Second  case:} $k<\infty$. From Lemma \ref{iterimp}, for all $i =  0, \ldots, k$, we have  the existence of vectors $l_i$ such that 
\begin{equation}\label{iteratio2}
		\underset{(y, t)\in Q^{(1-\delta)^i}_{\rho^{i}}}{\osc} \, (u(y,t)-l_{i}\cdot y)\leq \eta_1 \rho^i(1-\delta)^i,
		\end{equation}
		with 
		\begin{align*}
		&|l_{i}|\leq 2 (1-\delta)^{i}\qquad\text{for}\,\, i=0, \ldots,k-1,\\
		&|l_k-l_{k-1}|\leq C_1(1-\delta)^{k-1},
		\end{align*}
		and
		\begin{equation}\label{lisboa}
		2(1-\delta)^k\leq |l_k|\leq B(1-\delta)^{k-1}\leq 2B(1-\delta)^k:=M(1-\delta)^k.
		\end{equation}
		In particular,
\begin{equation}\label{porto}
		\underset{(y, t)\in Q^{(1-\delta)^k}_{\rho^k}}{\osc} \, (u(y,t)-l_{k}\cdot y)\leq\eta_1  \rho^k (1-\delta)^k.
		\end{equation}
For $(x,t)\in Q_1$ we define the rescaled function
		$$w(x,t):=\frac{u(\rho^k x, \rho^{2k}(1-\delta)^{-k\gamma} t)-l_k\cdot \rho^k x}{\rho^k(1-\delta)^k}.$$
	From \eqref{porto} we have
 $$\underset{(x,t)\in Q_{1}}{\osc}\, (w(x,t))\leq \eta_1.$$ 
 Moreover, $w$ satisfies
		$$\partial_t w-|Dw+K|^\gamma\left[\Delta u+(p-2) \dfrac{D^2 w (Dw +K), (\ Du+K)}{|Dw+K|^2}\right]=\tilde f,$$
where		$K:=l_{k}(1-\delta)^{-k}$ and
 $$\bar f(x, t):= \rho^k(1-\delta)^{-k(1+\gamma)}f(\rho^k x, \rho^{2k}(1-\delta)^{-k\gamma} t).$$
  Since we have $|l_k|\geq 2(1-\delta)^k$, it follows from \eqref{lisboa} that $|K|\geq 2$. The upper bound on $l_k$ implies that $|K|\leq M=M(p,n,\gamma)$. Notice also that $\norm{\tilde f}_{L^\infty(Q_1)}\leq\eps_0\leq \eta_1$.
From Lemma \ref{toulouse}, we have that the Lipschitz bound of $w$ is bounded by 1 in $Q_{7/8}$. Now we consider  for $(x,t)\in Q_1$ the function
		$$v(x,t):=\frac{u(\rho^k x, \rho^{2k}(1-\delta)^{-k\gamma} t)}{\rho^k(1-\delta)^k}=w(x,t)+K\cdot x.$$
 Then $v$ satisfies  
 $$\partial_t v=|Dv|^\gamma\Delta_p^N v+ \bar f.$$
 Moreover, $v$ is differentiable a.e., and for $(x,t)\in Q_{7/8}$ we have
 \begin{align*}|Dv(x,t)|=|Dw+K|&\geq |K|-|Dw|\geq 1,\\
 |D v(x,t)|&\leq 1+M.
 \end{align*}

We notice that $v$ solves an equation which is smooth in the gradient variables and uniformly parabolic, with ellipticity constants depending only on $p,n, \gamma$. It follows from \cite[Theorem 1.1]{LU68}  and \cite[ Lemma 12.13]{Lieb96} 
 that $v\in C^{1+\bar \alpha, (1+\bar\alpha)/2}_{loc}(Q_{3/4})$ for some $\bar \alpha = \bar \alpha(p,n,\gamma)>0$ and 
$$\norm{Dv}_{C^{\bar\alpha}(\omega)}\leq C\left(p,n, \gamma, \dist(\omega, \partial_pQ_{3/4})\right)$$
for any $\omega\subset \subset Q_{3/4}$.
Now, let $ 0<\alpha\leq\min\left(\bar\alpha, \frac{\log(1-\delta)}{\log(\rho)}\right)$.
Then, there exists a vector $l\in \R^n$ such that  in $Q_\rho^{1-\delta}$ we have
$$|Dv(x, t)-l|\leq C(|x|^{ \alpha}+|t|^{\frac{\alpha}{2}})\leq C(|x|^{ \alpha}+|t|^{\beta}),$$
and 
$$|v(x,t)-v(x,0)|\leq C|t|^{\frac{1+\alpha}{2}} \leq C|t|^\sigma.$$
(Recall that $\beta= \frac{\alpha}{2-\alpha \gamma}$ and $\sigma=\frac{1+\alpha}{2-\alpha \gamma}$.)

Recalling the definition of $v$, we get that in $Q_{\rho^{k+1}}^{(1-\delta)^{k+1}}$ it holds
\begin{align*}|Du(y,s)-(1-\delta)^k l|&\leq C(\rho^{-k \alpha}(1-\delta)^k|y|^{ \alpha}+(1-\delta)^k((\rho^{-2}(1-\delta)^{\gamma})^{k \beta}|s|^\beta)\\
&\leq C(|y|^{\alpha}+|s|^{\beta}),
\end{align*}
and 
\begin{equation}\label{bernard}
|u(y,s)-u(y,0)|\leq C\rho^k(1-\delta)^k (\rho^{-2}(1-\delta)^\gamma)^{k\sigma}|s|^\sigma\leq C|s|^\sigma,
\end{equation}
where we used that $\rho^{-\alpha}(1-\delta)\leq 1$ due to  $0<\alpha\leq \frac{\log{1-\delta)}}{\log{\rho}}$.

The gradient regularity part is completed by combining these estimates with \eqref{iteratio2}. Indeed, we have showed that  for $$0<\alpha=\min \left(\bar \alpha, \frac{\log(1-\delta)}{\log(\rho)}\right),$$ there exists a constant $C=C(p,n, \gamma)$ such that, for any $r\leq \frac{1}{2}$, there exists a vector $V=V(r)$ such that
\begin{equation*}
|u(x,t)-u(0,0)-V\cdot x|\leq Cr^{1+\alpha},
\end{equation*}
whenever $|x|+|t|^{\frac{1}{2-\alpha \gamma}}\leq r$.
Then the regularity of $Du$ follows from an easy adaptation of \cite[Lemma 12.12]{Lieb96} or \cite[Appendix]{AP17}. The H\"older regularity  of $u$ in time follows from \eqref{iteratio2},\eqref{lisboa},\eqref{porto} and  \eqref{bernard}. Indeed,  for $i=0,\ldots k$, we have 
\begin{align}\label{yohan}
\underset{(y, t)\in Q^{(1-\delta)^i}_{\rho^{i}}}{\osc} \, (u(y,t)-u(0,0))&\leq  \underset{(y, t)\in Q^{(1-\delta)^i}_{\rho^{i}}}{\osc} \, (u(y,t)-l_{i}\cdot y)+\underset{(y, t)\in Q^{(1-\delta)^i}_{\rho^{i}}}{\osc} \, l_i\cdot y\nonumber\\
&\leq  \eta_1(1-\delta)^i+2|l_i|\rho^i\\
&\leq C(p,n, \gamma)\rho^i(1-\delta)^i.\nonumber
\end{align}
The proof is completed by putting together estimates \eqref{yohan} and \eqref{bernard}. One gets that for $-1/4\leq t\leq 0$, it holds
\begin{equation*}
|u(0,t)-u(0,0)|\leq C(p,n, \gamma)|t|^\sigma.\qedhere
\end{equation*}
\end{proof}

\section{A better control on the Lipschitz estimates for deviation from planes}\label{ishii}
In this section, our aim is to provide a proof for Lemma \ref{toulouse}. We start with a suitable control on the Hölder norm of $w$ and apply again the Ishii-Lions's method in order to get good enough Lipschitz estimates.
\begin{lemma}
Let $K\in \R^n$ with $2\leq |K|\leq M$ for some $M=M(p,n, \gamma)>0$.
Let $w$ be a viscosity solution to 
$$\partial_t w-|Dw+K|^\gamma\left[\Delta u+(p-2) \dfrac{D^2 w (Dw +K), (\ Du+K)}{|Dw+K|^2}\right]=\tilde f.$$
with $w(0)=0$. There exists $\eta_0=\eta_0(p,n,\gamma)$ such that if $\norm{w}_{L^\infty(Q_1)}\leq \eta_0$ and $\norm{\tilde f}_{L^\infty(Q_1)}\leq \eta_0$, then $w$ is locally  Hölder  continuous in space and for $(x,t), (y,t)\in Q_{7/8}$ it holds
$$|w(x,t)-w(y,t)|\leq |x-y|^{\bar \beta},$$
where $\bar\beta =\min(\frac12, \dfrac{4}{5\bar C}M^{\frac{\gamma}{1+\gamma}})$, with $\bar C$ being the constant coming from \eqref{levicopar}.
\end{lemma}

\begin{proof}
We fix $x_0, y_0\in B_{7/8}$, $t_0\in (-(7/8)^2,0)$ and consider the function
		\begin{align*}
		\Phi(x, y,t):&=w(x,t)-w(y,t)-L_2\abs{x-y}^{\bar \beta}-\frac{L_1}{2}\abs{x-x_0}^2-
		\frac{L_1}{2}\abs{y-y_0}^2-\frac{L_1}{ 2} (t-t_0)^2,
		\end{align*}
		where 
\[
L_1=4^5\norm{w}_{L^\infty(Q_1)},
\]
\begin{align}
L_2&=\left(\frac{10}{\bar\beta}+5+\dfrac{(64 n|p-2)|+8n \max(1, p-1)+16M^{-\gamma})(3-\bar\beta)}{\min(1, p-1)\bar\beta(1-\bar\beta)}\right)L_1\nonumber\\
&\qquad +\dfrac{1}{\bar\beta}(32 n(1+|p-2|)(3-\bar\beta))^{-1/\gamma}L_1+\dfrac{16M^{-\gamma}(3-\bar\beta)}{\min(1, p-1)\bar\beta(1-\bar\beta)}	\norm{\tilde f}_{L^\infty(Q_1)}\nonumber\\
& =C\left(\norm{w}_{L^\infty(Q_1)}+\norm{\tilde f}_{L^\infty(Q_1)}\right)\nonumber\\
&\leq \dfrac{1}{2\eta_0}\left(\norm{w}_{L^\infty(Q_1)}+\norm{\tilde f}_{L^\infty(Q_1)}\right).\nonumber
		\end{align}
We want to show that $\Phi(x, y,t)\leq 0$ for $(x,y)\in \overline{B_1}\times\overline{ B_1}$ and $t\in [-1,0]$.  We argue by contradiction.  We assume that
		$\Phi$ has a positive
		maximum at some point $(\bar x, \bar y,\bar t)\in \bar B_{1}\times \bar B_{1}\times [-1,0]$ and we are going to get a contradiction.
The positivity of the maximum of $\Phi$  implies that $\bar x\neq \bar y$ and 
		\begin{align*}
|\bar y- y_0|,\, |\bar x-x_0|,\, |\bar t-t_0|\leq \sqrt{4\norm{w}_{L^\infty(Q_1)}/L_1}=\sqrt{4/4^5}\leq 1/16.
		\end{align*}
It follows that $\bar x$ and $\bar y$ are in $B_{15/16}$ and $\bar t\in (-(15/16)^2,0)$. \\		
		From the Lipschitz regularity of $w$ and the estimate \eqref{levicopar},  we have for some $\bar C=\bar C(p,n, \gamma)$
		$$L_2|\bar x-\bar y|^{\bar\beta}\leq \bar C|K|^{\frac{1}{1+\gamma}}|\bar x-\bar y|.$$
Using that (by hypothesis) $\bar\beta \bar C\leq \dfrac45M^{\frac{\gamma}{1+\gamma}}$ and $|K|\leq M$, we have
		\begin{equation}\label{paris}
		\bar\beta L_2|\bar x-\bar y|^{\bar\beta-1}\leq\bar\beta  \bar C|K|^{\frac{1}{1+\gamma}}\leq \frac45 |K|.
		\end{equation}
	\noindent {\bf Step 1.}
		The  Jensen-Ishii's lemma (see \cite[Theorem 8.3]{crandall1992user}) ensures the existence of
		
		\[
		\begin{split}
		&(b+L_1(\bar t-t_0), a_1,X+L_1I)\in \overline{\mathcal{P}}^{2,+} w(\bar x, \bar t),\\ &(b, a_2,Y-L_1I)\in \overline{\mathcal{P}}^{2,-} w(\bar y, \bar t),
		\end{split}
		\]
		 where
		\[
		\begin{split}
		a_1=L_2\varphi'(|\bar x-\bar y|) \frac{\bar x-\bar y}{\abs{\bar x-\bar y}}+L_1(\bar x-x_0),\qquad
		a_2=L_2\varphi'(|\bar x-\bar y|) \frac{\bar x-\bar y}{\abs{\bar x-\bar y}}-L_1(\bar y-y_0).
		\end{split}
		\]
		
	Using that  $L_2\geq \dfrac{2}{\bar\beta} L_1$, we have for $i\in \{1,2\}$,
		\begin{equation}\label{portorico}
		3L_2\bar\beta \abs{\bar x-\bar y}^{\bar\beta-1}\geq\abs{a_i}\geq \frac{L_2}{2} \bar\beta \abs{\bar x-\bar y}^{\bar\beta-1}.
		\end{equation}

		\noindent Because of Jensen-Ishii's lemma \cite[Theorem 12.2]{crandnote},  we can take $X, Y\in \mathcal{S}^n$ such that for any  $\tau>0$ such that $\tau Z<I$, it holds\\
		\begin{equation}\label{maineq1}
		-\frac{2}{\tau} \begin{pmatrix}
		I&0\\
		0&I 
		\end{pmatrix}\leq
		\begin{pmatrix}
		X&0\\
		0&-Y 
		\end{pmatrix}\leq \begin{pmatrix} Z^\tau& -Z^\tau\\
		-Z^\tau& Z^\tau\end{pmatrix},
		\end{equation}
	
		where 
		
		\begin{align*}	
		Z=L_2\bar\beta\abs{\bar x-\bar y}^{\bar\beta-2}\left(I+(\bar\beta-2)\frac{\bar x-\bar y}{\abs{\bar x-\bar y}}\otimes \frac{\bar x-\bar y}{\abs{\bar x-\bar y}}\right),\qquad Z^\tau&= (I-\tau Z)^{-1}Z.
		\end{align*}

		\noindent We choose $\tau=\dfrac{1}{2L_2\bar\beta\abs{\bar x-\bar y}^{\bar\beta-2}}$ so  that we have
		\begin{align*}	
		Z^\tau=(I-\tau Z)^{-1} Z=2L_2\bar\beta\abs{\bar x-\bar y}^{\bar\beta-2}\left(I-2\frac{2-\bar\beta}{3-\bar\beta} \frac{\bar x-\bar y}{\abs{\bar x-\bar y}}\otimes \frac{\bar x-\bar y}{\abs{\bar x-\bar y}}\right).
\end{align*}	
It follows that for $\xi=\frac{\bar x-\bar y}{\abs{\bar x-\bar y}}$, 
	
				\begin{equation}\label{oufi}
		\langle Z^\tau \xi,\xi\rangle= 2L_2\bar\beta\abs{\bar x-\bar y}^{\bar\beta-2}\left(\frac{\bar\beta-1}{3-\bar\beta}\right)<0.
		\end{equation}
Applying the inequality \eqref{maineq1} to any  vector $(\xi,\xi)$ with $\abs{\xi}=1$, we get that $X- Y\leq 0$ and 
		\begin{equation}\label{gilout}
		\norm{X},\norm{Y}\leq 4L_2\bar\beta\abs{\bar x-\bar y}^{\bar\beta-2}.
		\end{equation}
Setting $\xi_1=a_1+K$, $\xi_2=a_2+K$, we get by using \eqref{portorico},\eqref{paris},    
		\begin{align}\label{koivu1}
		2|K|\geq\abs{\xi_i}&\geq \abs{K}-3\bar\beta L_2|\bar x-\bar y|^{\bar\beta-1}\geq \frac{L_2}{4} \bar\beta \abs{\bar x-\bar y}^{\bar\beta-1}.
		\end{align}

\noindent{\bf Step 2.}		We write the  viscosity inequalities
		\begin{equation*}
		\begin{split}
	(L_1(\bar t-t_0)+b) -\tilde f(\bar x,\bar t)&\leq |\xi_1|^\gamma\left[ \tr (X+L_1I)+(p-2)\dfrac{\left\langle(X+L_1I) \xi_1, \xi_1\right\rangle}{|\xi_1|^2}\right]\\
		b-\tilde f(\bar y,\bar t)&\geq |\xi_2|^\gamma\left[ \tr(Y-L_1I)+(p-2) \dfrac{\left\langle(Y-L_1I) \xi_2, \xi_2\right\rangle}{|\xi_2|^2}\right].
		\end{split}
		\end{equation*}
	For $\eta \neq 0$, denote   $\hat \eta=\dfrac{\eta}{|\eta |}$ and  $\A(\eta):= I+(p-2)\hat\eta\otimes \hat\eta$.
		Assume that $|\xi_1|\geq |\xi_2|$ (the other case can be treated similarly).
Adding the two inequalities and using  that $|\bar t-t_0|\leq 2$, we get  
		\begin{align}\label{gregory1}
		-(2L_1+2||\tilde f||_{L^\infty(Q_1)} )|\xi_1|^{-\gamma}\leq  &\underbrace{\tr (\A(\xi_1)(X-Y))}_{(i_1)}\nonumber	+\underbrace{ \tr ((\A(\xi_1)-\A(\xi_2))Y)}_{(i_2)}\nonumber\\
		&+\underbrace{(|\xi_1|^\gamma-|\xi_2|^\gamma)\tr(\A(\xi_2)Y)|\xi_1|^{-\gamma}}_{(i_3)}\nonumber\\
		&+\underbrace{L_1\big[\tr (\A(\xi_1))+(|\xi_2||\xi_1|^{-1})^\gamma\tr (\A(\xi_2))}_{(i_4)} \big].
		\end{align}
		
We start with an estimate  for ($i_1$).  Notice   $X-Y\leq 0$   (this follows from  \eqref{maineq1}) and that  at least one of the eigenvalues of $X-Y$ is negative and smaller than $8L_2\bar\beta\abs{\bar x-\bar y}^{\bar\beta-2}\left(\frac{\bar\beta-1}{3-\bar\beta}\right)$.
Using that the eigenvalues of $\A(\xi_1)$ belong to $[\min(1, p-1), \max(1, p-1)]$ , we get
\begin{align*}  
\tr(\A(\xi_1) (X-Y))&\leq \sum_i \lambda_i(\A(\xi_1))\lambda_i(X-Y)\leq  8L_2\bar\beta\abs{\bar x-\bar y}^{\bar\beta-2}\left(\frac{\bar\beta-1}{3-\bar\beta}\right)\min(1,p-1).
		\end{align*}
We estimate  ($i_2$) by 
\begin{align*}
		\tr( (\A(\xi_1)-\A(\xi_2)) Y)
		&\leq 2n\abs{p-2}\norm{Y}|\hat \xi_1-\hat\xi_2|.
		\end{align*}
		Using that $|\xi_1-\xi_2|\leq  L_1$ and  the estimate \eqref{koivu1}, we get 
		\begin{equation*}
		\begin{split}
		\abs{\hat \xi_1-\hat \xi_2}
\le \max\left( \frac{\abs{\xi_2- \xi_1}}{\abs{\xi_2}},\frac{ \abs{\xi_2- \xi_1}}{\abs{\xi_1}}\right)\le \frac {4  L_1}{ \bar\beta L_2 \abs{\bar x-\bar y}^{\bar\beta-1}}.
		\end{split}
		\end{equation*}
	Recalling \eqref{gilout}, it follows that 
	\begin{align*}
	|\tr( (\A(\xi_1)-\A(\xi_2)) Y)|&\leq 32n\abs{p-2}L_1 \abs{\bar x-\bar y}^{-1}.
	 \end{align*}
	 Next we estimate the term ($i_3$). We have $|\xi_2|/|\xi_1|\leq 2$ and $|\xi_1-\xi_2|\leq L_1$.  Using  the mean value theorem and the estimate \eqref{koivu1},  we get that  	
	\begin{align*}
	|\xi_1|^{-\gamma} ||\xi_1|^\gamma-|\xi_2|^\gamma|
	  &\leq |\xi_1-\xi_2|^{-\gamma} |\xi_2|^{\gamma} \leq L_1^{-\gamma} \left(\frac 14 L_2 \bar\beta \abs{\bar x-\bar y}^{\bar\beta-1}\right)^{\gamma}.
	\end{align*}
	
We obtain
	\begin{align}\label{nousa}
	   |\xi_1|^{-\gamma}||\xi_1|^\gamma -|\xi_2|^\gamma|| \tr(\A(\xi_2)Y)|\leq 16nL_2^{1+\gamma}\bar\beta^{1+\gamma}\abs{\bar x-\bar y}^{\bar\beta-2+\gamma(\bar\beta-1)}(1+|p-2|)L_1^{-\gamma} 
	 \end{align}
We estimate ($i_4$) by using the estimate \eqref{koivu1}, and get 
		$$ L_1(\tr(\A(\xi_1)+(|\xi_2||\xi_1|^{-1})^\gamma\tr(\A(\xi_2)))\leq  3L_1n\max(1, p-1).$$
Finally, we gather the previous estimates and plug them into \eqref{gregory1}. We get 
		\begin{align*}
		0&\leq 2(L_1 + \norm{\tilde f}_{L^\infty(Q_1)})(2M)^{-\gamma}+3L_1n\max(1, p-1)\\
		& \quad+8\min(1, p-1) L_2\bar\beta\abs{\bar x-\bar y}^{\bar\beta-2}\left(\frac{\bar\beta-1}{3-\bar\beta}\right)+32 n\abs{p-2}L_1 \abs{\bar x-\bar y}^{-1}\\
		&\quad+  16nL_2^{\gamma+1}\bar\beta^{\gamma+1}\abs{\bar x-\bar y}^{\bar\beta-2+\gamma(\bar\beta-1)}(1+|p-2|)L_1^{-\gamma}. 
		\end{align*}
		Denote  $H:=\min(1, p-1)L_2\dfrac{\bar\beta(1-\bar\beta)}{3-\bar\beta}|\bar x-\bar y|^{\bar\beta-2}$. Using the definition of $L_2$, we have 
		 \begin{align*}
		H &\geq \max \Big( 32 n|p-2|L_1|\bar x-\bar y|^{-1},3n L_1\max(1, p-1),\\
		& \quad 2(L_1+\norm{\tilde f}_{L^\infty(Q_1)})(2M)^{-\gamma},16nL_2^{\gamma+1}\bar\beta^{\gamma+1}\abs{\bar x-\bar y}^{\bar\beta-2+\gamma(\bar\beta-1)}(1+|p-2|)L_1^{-\gamma}\Big).
				 \end{align*}
		 It follows that 
		 $$0\leq4\min(1, p-1) L_2\bar\beta\abs{\bar x-\bar y}^{\bar\beta-2}\left(\frac{\bar\beta-1}{3-\bar\beta}\right)<0$$
and we get a contradiction. Hence,  $\Phi(x,y,t)\leq 0$ in $Q_1$.
We conclude the proof by using that for any  $x_0,y_0\in B_{7/8}$ and $t_0\in (-(7/8)^2, 0]$, we have $\Phi(x_0,y_0, t_0)\leq 0$ so that  we get
		\[
		| w(x_0, t_0)- w(y_0, t_0)|\leq L_2|x_0-y_0|^{\bar\beta}.\]
		Using that $L_2\leq 1$ if $\norm{w}_{L^\infty(Q_1)}\leq \eta_0$ and $\norm{\tilde f}_{L^\infty(Q_1)}\leq \eta_0$, we get the desired estimate.
\end{proof}
\subsection{Proof of Lemma \ref{toulouse}}

 We  fix $x_0, y_0\in B_{3/4}$, $t_0\in (-(3/4)^2,0)$. Let
 \begin{align*}
 \nu=1-\gamma\bar\beta/2\in (1,2),
 \end{align*}
 where $\bar \beta$ is given by the previous lemma. Define
 \[
		\begin{split}
		\vp(s)=
		\begin{cases}
		s-s^{\nu}\kappa_0& 0\le s\le s_1:=(\frac 1 {\nu\kappa_0})^{1/(\nu-1)}  \\
		\vp(s_1)& \text{for}\ s\geq s_1,
		\end{cases}
		\end{split}
		\]
		where $\kappa_0>0$ is taken so that  $s_1>2 $ and $\nu \kappa_0s_1^{\nu-1}\leq 1/4$. With these choices we have $\varphi'(s)\in  [\frac34,1]$  and $\varphi''(s)<0$ when $s\in (0,2]$.
 Let
 \begin{align*}
  L_1&=4^5\norm{w}_{L^\infty(Q_1)}=C\norm{w}_{L^\infty(Q_1),}\\
 L_2&=L_1+\dfrac{32n\max(1, p-1)L_1}{\nu(\nu-1)\kappa_0\min(1, p-1)}+\dfrac{32(L_1+\norm{\tilde f}_{L^\infty(Q_1)})M^{-\gamma}}{\nu(\nu-1)\kappa_0\min(1, p-1)}\\
 & \quad +\left(\dfrac{64n(1+4\kappa_0)(1+|p-2|)}{\nu(\nu-1)\kappa_0\min(1, p-1)}\right)^{-1/\gamma}\sqrt{L_1}+\dfrac{512n|p-2|(2+4\kappa_0)}{\nu(\nu-1)\kappa_0\min(1, p-1)}\sqrt{L_1}\\
 & \leq \frac{1}{4\eta_1}(\norm{w}_{L^\infty(Q_1)}+\norm{\tilde f}_{L^\infty(Q_1)})+\frac12 \frac{\sqrt{\norm{w}_{L^\infty(Q_1)}}}{\sqrt{\eta_1}}.
 \end{align*}
 We consider the function
		\begin{align*}
		\Phi(x, y,t):=w(x,t)-w(y,t)-L_2\vp(\abs{x-y})-\frac {L_1}{2}\abs{x-x_0}^2-\frac {L_1}{2}\abs{y-y_0}^2-\frac{ L_1}{ 2} (t-t_0)^2.
		\end{align*}
We want to show that $\Phi(x, y,t)\leq 0$ for $(x,y)\in \overline{B_1}\times\overline{ B_1}$ and $t\in [-1,0]$.  We proceed by contradiction assuming that
		$\Phi$ has a positive
		maximum at some point $(\bar x, \bar y,\bar t)\in \bar B_1\times \bar B_1\times [-1,0]$ and  we aim to get a contradiction.
		From the positivity of the maximum, we have $\bar x\neq \bar y$ and
		\begin{align*}
|\bar y- y_0|,\, |\bar x-x_0|,\, |\bar t-t_0|\leq \sqrt{4\norm{w}_{L^\infty(Q_1)}/L_1}=\sqrt{4/C}\leq 1/8.
		\end{align*}
It follows that $\bar x$ and $\bar y$ are in $B_{7/8}$ and $\bar t\in (-(7/8)^2,0)$. \\
From the Hölder regularity of $w$ and the hypothesis that $\norm{w}_{L^\infty(Q_1)}\leq\eta_1\leq\eta_0$,  we have
\begin{align}\label{coimbre}
L_1|\bar x-x_0|^2,\quad L_1 |\bar y-y_0|^2\leq 2|\bar x-\bar y|^{\bar\beta}.
\end{align}	
The  Jensen-Ishii's lemma gives the existence of
		\[
		\begin{split}
		&(b+L_1(\bar t-t_0),a_1,X+L_1I)\in \ol P^{2,+} w(\bar x,\bar t),\\ &(b,a_2,Y-L_1I)\in \ol P^{2,-}w(\bar y,\bar t),
		\end{split}
		\]
where\[
		\begin{split}
		a_1=L_2\vp'(|\bar x-\bar y|) \frac{\bar x-\bar y}{\abs{\bar x-\bar y}}+L_1(\bar x-x_0),\qquad
		a_2=L_2\vp'(|\bar x-\bar y|) \frac{\bar x-\bar y}{\abs{\bar x-\bar y}}-L_1(\bar y-y_0).
		\end{split}
		\]
Using that $\vp'\geq \frac34$ and $L_2\geq L_1$, we have
		\[
	2L_2\geq	\abs{a_1},\abs{a_2}\geq L_2\varphi'(|\bar x-\bar y|) - L_1\abs{\bar x-x_0}\ge 3/4 L_2-L_1/8\geq  \frac{L_2}{2}.
		\]
		Also, by  Jensen-Ishii's lemma, for any $\tau>0$, we can take $X, Y\in \mathcal{S}^n$ such that 
		\begin{equation}\label{matriceineq1}
		- \big[\tau+2\norm{Z}\big] \begin{pmatrix}
		I&0\\
		0&I 
		\end{pmatrix}\leq
		\begin{pmatrix}
		X&0\\
		0&-Y 
		\end{pmatrix}
		\end{equation}
		and
		\begin{equation}\label{matineq2}
			\begin{pmatrix}
		X&0\\
		0&-Y 
		\end{pmatrix}
		\le 
		\begin{pmatrix}
		Z&-Z\\
		-Z&Z 
		\end{pmatrix}
		+\frac2\tau \begin{pmatrix}
		Z^2&-Z^2\\
		-Z^2&Z^2 
		\end{pmatrix},
		\end{equation}
		where 
		\begin{align*}	
		Z=L_2\vp''(|\bar x-\bar y|) \frac{\bar x-\bar y}{\abs{\bar x-\bar y}}\otimes \frac{\bar x-\bar y}{\abs{\bar x-\bar y}} +\frac{L_2\vp'(|\bar x-\bar y|)}{\abs{\bar x-\bar y}}\Bigg( I- \frac{\bar x-\bar y}{\abs{\bar x-\bar y}}\otimes \frac{\bar x-\bar y}{\abs{\bar x-\bar y}}\Bigg)
		\end{align*}	
		and 
		\begin{align*}	
		Z^2=
		\frac{L_2^2(\vp'(|\bar x-\bar y|))^2}{\abs{\bar x-\bar y}^2}\Bigg( I- \frac{\bar x-\bar y}{\abs{\bar x-\bar y}}\otimes \frac{\bar x-\bar y}{\abs{\bar x-\bar y}}\Bigg)+L_2^2(\vp''(|\bar x-\bar y|))^2 \frac{\bar x-\bar y}{\abs{\bar x-\bar y}}\otimes \frac{\bar x-\bar y}{\abs{\bar x-\bar y}}.
\end{align*}	
We notice that
		\begin{align}\label{lilou}
		&\norm{Z}\leq L_2 \frac{\vp'(|\bar x-\bar y|)}{|\bar x-\bar y|},\qquad \norm{Z^2}\leq L_2^2\left(|\vp''(|\bar x-\bar y|)|+\dfrac{|\vp'(|\bar x-\bar y|)|}{|\bar x-\bar y|}\right)^2,
		\end{align}
		and for $\xi=\frac{\bar x-\bar y}{\abs{\bar x-\bar y}}$, we have
	
				\begin{equation*}
		\langle Z\xi,\xi\rangle=L_2\vp''(|\bar x-\bar y|)<0, \qquad\langle Z^2\xi,\xi\rangle=L_2^2(\vp''(|\bar x-\bar y|))^2.
		\end{equation*}
		We choose  $\tau=4L_2\left(|\vp''(|\bar x-\bar y|)|+\dfrac{|\vp'(|\bar x-\bar y|)|}{|\bar x-\bar y|}\right)$ and get that for $\xi=\frac{\bar x-\bar y}{\abs{\bar x-\bar y}}$,
			\begin{align}\label{mercit}
		\langle Z\xi,\xi\rangle +\frac2\tau \langle Z^2\xi,\xi\rangle=L_2\left(\vp''(|\bar x-\bar y|)+\frac2\tau L_2(\vp''(|\bar x-\bar y|))^2\right)\leq \dfrac{L_2}{2}\vp''(|\bar x-\bar y|).
		\end{align}
		From the inequalities \eqref{matriceineq1} and \eqref{matineq2}, we deduce that $X- Y\leq 0$ and $\norm{X},\norm{Y}\leq 2\norm{Z}+\tau$. Moreover,  applying the matrix inequality \eqref{matineq2} to the vector $(\xi,-\xi)$ where $\xi:=\frac{\bar x-\bar y}{|\bar x-\bar y|}$  and using \eqref{mercit}, 
		we get
		\begin{align}\label{camille}
		\langle (X-Y) \xi, \xi\rangle\leq 4\left(\langle Z\xi,\xi\rangle+\frac2\tau\langle Z^2\xi,\xi\rangle\right)\leq 2 L_2\vp''(|\bar x-\bar y|)<0.
		\end{align}
Hence, at least one of the eigenvalue of $X-Y$ is negative and smaller than $2 L_2\vp''(|\bar x-\bar y|)$. Now, setting $\xi_1=a_1+K$, $\xi_2=a_2+K$ and  using that for $\norm{w}_{L^\infty(Q_1)}, \norm{\tilde f}_{L^\infty(Q_1)}\leq \eta_1$  we have $L_2\leq 1$, it holds
		\begin{align}\label{koivu}
		2|K|\geq \abs{\xi_i}\geq|K|- \abs{a_i}\geq \frac{\abs{a_1}}{2}\geq \frac{L_2}{4}.
		\end{align}
		Writing the  viscosity inequalities and adding them, we end up with
					\begin{align}\label{gregory1f}
		-2|\xi_1|^{-\gamma}(L_1+||\tilde f||_{L^\infty(Q_1)}) \leq  &\underbrace{\tr (\A(\xi_1)(X-Y))}_{(I)}		+\underbrace{ \tr ((\A(\xi_1)-\A(\xi_2))Y)}_{(II)}\nonumber\\
		&+\underbrace{|\xi_1|^{-\gamma}(|\xi_1|^\gamma-|\xi_2|^\gamma)\tr(\A(\xi_2)Y)}_{(III)}\nonumber\\
		&+\underbrace{L_1\big[\tr (\A(\xi_1))+(|\xi_2||\xi_1|^{-1})^\gamma\tr (\A(\xi_2))}_{(IV)} \big].\nonumber
		\end{align}
Next we estime these terms separately. Using \eqref{camille}, we estimate 
		\begin{align*}  
		\tr(\A(\xi_1) (X-Y))\leq \sum_i \lambda_i(\A(\xi_1))\lambda_i(X-Y)\leq- 2\min(1, p-1) L_2 \nu(\nu-1)\kappa_0|\bar x-\bar y|^{\nu-2}.
		\end{align*}
		We estimate $(II)$ by
		$\tr( (\A(\xi_1)-\A(\xi_2)) Y)
		\leq 2n\abs{p-2}\norm{Y}|\ol \xi_1-\ol\xi_2|$.
Using \eqref{coimbre}, we have
		\begin{equation*}
		\begin{split}
		\abs{\ol \xi_1-\ol \xi_2}=
		\abs{\frac{\xi_1}{\abs {\xi_1}}-\frac{\xi_2}{\abs {\xi_2}}}
\le \max\left( \frac{\abs{\xi_2- \xi_1}}{\abs{\xi_2}},\frac{ \abs{\xi_2- \xi_1}}{\abs{\xi_1}}\right)\le \dfrac{16\sqrt{L_1}|\bar x-\bar y|^{\bar\beta/2}}{L_2}\
		\end{split}
		\end{equation*}
		where we used \eqref{koivu}.
		Using  \eqref{matriceineq1}--\eqref{lilou}, we have
		\begin{equation*}
		\norm{Y}=\max_{\ol \xi} |\langle Y\ol \xi, \ol \xi\rangle|
		\le 2 |\langle Z\ol \xi,\ol \xi \rangle|+\frac4\tau|\langle Z^2\ol \xi,\ol \xi \rangle| \leq 4L_2\left( |\bar x-\bar y|^{-1}+ \nu(\nu-1)\kappa_0|\bar x-\bar y|^{\nu-2}\right).
		\end{equation*}
	We get	
	\begin{align*}
	 (II)&\leq 128n\abs{p-2}\sqrt{L_1}|\bar x-\bar y|^{\bar\beta/2}\left( |\bar x-\bar y|^{-1}+ \nu(\nu-1)\kappa_0|\bar x-\bar y|^{\nu-2}\right).
	 \end{align*}
The mean value theorem and the estimate \eqref{koivu} imply that
\begin{align*}
|\xi_1|^{-\gamma}	 ||\xi_1|^\gamma-|\xi_2|^\gamma|&\leq |\xi_1-\xi_2|^{-\gamma}|\xi_2|^{\gamma}\leq 4(\sqrt{L_1}|\bar x-\bar y|^{\bar\beta/2})^{-\gamma}L_2^{\gamma}.
\end{align*}
Consequently, it holds
\begin{align}
(III)\leq  4nL_2\left( |\bar x-\bar y|^{-1}+ \nu(\nu-1)\kappa_0|\bar x-\bar y|^{\nu-2}\right)(1+|p-2|)4(\sqrt{L_1}|\bar x-\bar y|^{\bar\beta/2})^{-\gamma}L_2^{\gamma}.
	 \end{align}
The last term (IV) is easy to estimate by $$ (IV)\leq 2n L_1n\max(1, p-1).$$
	Summing up all the estimates,  it holds
		\begin{align*}
		0&\leq 2(2M)^{-\gamma}(L_1+\norm{\tilde f}_{L^\infty(Q_1)})+2n L_1\max(1, p-1)\\
		 &\quad+  16nL_2^{1+\gamma}\left( \abs{\bar x-\bar y}^{-1}+ \nu(\nu-1)\kappa_0|\bar x-\bar y|^{\nu-2}\right)(1+|p-2|)(\sqrt{L_1}|\bar x-\bar y|^{\bar\beta/2})^{-\gamma}\\
		 &\quad
	+ 128n\abs{p-2}\sqrt{L_1}|\bar x-\bar y|^{\bar\beta/2}\left( \abs{\bar x-\bar y}^{-1}+ \nu(\nu-1)\kappa_0|\bar x-\bar y|^{\nu-2}\right)\\
	&\quad- 2\min(1, p-1) L_2(\nu-1)\nu \kappa_0|\bar x-\bar y|^{\nu-2}.
	\end{align*}
	Recall that $\nu=1-\gamma\frac{\bar\beta}{ 2}\in (1,2)$ and denote  $\bar H:=\dfrac{1}{4}\min(1, p-1)L_2(\nu-1)\nu \kappa_0|\bar x-\bar y|^{\nu-2}$. Using the definition of $L_2$ and the fact that $1-\nu-\gamma\bar\beta/2\geq 0$, $1-\nu+\bar\beta/2\geq 0$, we have
		 \begin{align*} 
		\bar H&\geq		 2n L_1\max(1, p-1)\\
		\bar H&\geq 2(L_1+\norm{\tilde f}_{L^\infty(Q_1)})(2M)^{-\gamma}\\
		\bar H&\geq 16nL_2^{\gamma+1}\left( \abs{\bar x-\bar y}^{-1}+ \nu(\nu-1)\kappa_0|\bar x-\bar y|^{\nu-2}\right)(1+|p-2|)(\sqrt{L_1}|\bar x-\bar y|^{\bar\beta/2})^{-\gamma}\\
		\bar H&\geq 128n\abs{p-2}\sqrt{L_1}|\bar x-\bar y|^{\bar\beta/2}\left( \abs{\bar x-\bar y}^{-1}+ \nu(\nu-1)\kappa_0|\bar x-\bar y|^{\nu-2}\right).
		 \end{align*}
		 It follows that 
		 $$0\leq -\min(1, p-1)L_2(\nu-1)\nu \kappa_0|\bar x-\bar y|^{\nu-2}<0,$$
		 which is  a contradiction.   Hence,  $\Phi(x,y,t)\leq 0$ in $Q_1$.
		 This concludes the proof  since for any  $x_0,y_0\in B_{3/4}$ and $t_0\in (-(3/4)^2, 0]$, we have $\Phi(x_0,y_0, t_0)\leq 0$ and  we get
		\[
		|w(x_0, t_0)-w(y_0, t_0)|\leq L_2|x_0-y_0|.
		\]
		Using that $L_2\leq 1$ if $\norm{u}_{L^\infty(Q_1)}\leq \eta_1$ and $\norm{\tilde f}_{L^\infty(Q_1)}\leq \eta_1$, we get the desired estimate.\qed
\begin{remark}
If one could adapt the result of  Wang \cite{wangyu} (see also the work of Savin \cite{savin} in the elliptic case) and prove that small perturbation of smooth  solutions to  some uniformly parabolic equation with a small enough continuous source term are  locally $C^{1, \alpha}$, then the proof will proceed without those Lipschitz estimates. This was done in \cite{silva19} for equation 
$$\partial_t u -F(x,t,D^2 u)=f(x,t),$$ where the operator is uniformly parabolic, and the source term is continuous. A possible generalization of \cite{silva19} in case of singular or degenerate operators remains to be done.
\end{remark}


\begin{thebibliography}{10}




\bibitem{arg}
{\sc E. Argiolas, F.  Charro, I. Peral},
On the Aleksandrov-Bakel'man-Pucci estimate for some elliptic and parabolic nonlinear operators. {\em 
Arch. Ration. Mech. Anal.} 202(3):875--917, 2011. 
\bibitem{att19}
{\sc A. Attouchi},
\newblock  {Local regularity for quasi-linear parabolic equations in non-divergence form}.
{\em ArXiv:1809.03241}, 2019.
\bibitem{AP17}
{\sc A. Attouchi, M. Parviainen},
\newblock H\"older regularity for the gradient of the inhomogeneous parabolic normalized $p$-Laplacian.
{\em  Commun. Contemp. Math.},  20 no. 4, 1750035, 27 pp, 2018.

\bibitem{Attouchipr17}
{ \sc A. Attouchi, M. Parviainen, E. Ruosteenoja}, $C^{1,\alpha}$ regularity for the normalized $p$-Poisson problem. \emph{J. Math. Pures Appl.}, 108(4):553--591, 2017.

\bibitem{Attouchir18}
{ \sc A. Attouchi, E. Ruosteenoja}, Remarks on regularity for p-Laplacian type equations in non-divergence form. \emph{J. Differential Equations}, 265(5):1922--1961, 2018.

\bibitem{Banerjeeg15}
{ \sc A. Banerjee, N. Garofalo}, On the Dirichlet boundary value problem for the normalized p-Laplacian evolution. \emph{Commun. Pure Appl. Anal.}, 14:1--21, 2015.

\bibitem{BM19} {\sc A. Banerjee, I.H. Munive}, Gradient continuity estimates for normalized $p$-Poisson equation. To appear in {\em  Commun. Contemp. Math.} Preprint in {\em ArXiv:1904.13076}, 2019.

\bibitem{Bertim19}
{ \sc D. Berti, R. Magnanini}, Short-time behavior for game-theoretic p-caloric functions. \emph{J. Math. Pures Appl.}, 126(9):249--272, 2019.







\bibitem{bata}
{\sc T. Bhattacharya, L. Marazzi},  On the viscosity solutions to a class of nonlinear degenerate parabolic differential equations. {\em Rev. Mat. Complut.} 30(3):621--656, 2017. 



\bibitem{Birindellid10}
{\sc I. Birindelli, F. Demengel}, Regularity and uniqueness of the first eigenfunction for singular fully nonlinear operators. {\em J. Differential Equations}, 249(5):1089--1110, 2010.


\bibitem{Birindellid14}
{\sc I.  Birindelli,  F. Demengel}, $ C^{1,\beta}$ regularity for Dirichlet problems associated to fully nonlinear degenerate elliptic equations. {\em ESAIM Control Optim. Calc. Var.} 20(4):1009--1024, 2014.






\bibitem{crandnote}
{\sc M. G. Crandall},
\newblock
Viscosity solutions: a primer. Viscosity solutions and applications (Montecatini Terme, 1995). {\em 
Lecture Notes in Math. 1660, Springer, Berlin}, 1--43, 1997. 
\bibitem{crandall1992user}
{\sc M.G. Crandall, H.~Ishii,  P-L Lions},
\newblock User's guide to viscosity solutions of second order partial
  differential equations.
\newblock {\em Bull. Am. Math. Soc.}  27(1):1--67, 1992.
\bibitem{silva19}
{\sc J. V.  da Silva, D. Dos Prazeres}, Schauder type estimates for "flat'' viscosity solutions to non-convex fully nonlinear parabolic equations and applications.
 {\em Potential Anal.} 50(2):149--170, 2019.


\bibitem{dem11}
{\sc  F. Demengel},
\newblock
Existence's results for parabolic problems related to fully nonlinear operators degenerate or singular.
\newblock{\em Potential Anal.} 35(1):1--38, 2011.


\bibitem{dib93}
{\sc  E. DiBenedetto}, Degenerate parabolic equations. {\em Springer, New York}, 1993.
\bibitem{df85}
{\sc E. DiBenedetto,  A. Friedman},  H\"older estimates for nonlinear degenerate parabolic systems. {\em J. Reine Angew. Math.}  357:1--22, 1985.

\bibitem{does11}
{ \sc K.~Does},
\newblock An evolution equation involving the normalized $p$-{L}aplacian.
\newblock {\em C{P}{A}{A}}, 10(1):361--396, 2011.

\bibitem{Dong19}
{ \sc H. Dong, P. Fa, Y.R.Y. Zhang, Y. Zhou},
Second order regularity for elliptic and parabolic equations involving $p$-Laplacian via a fundamental inequality. {\em arXiv:1908.01547}, 2019.

\bibitem{Han18}
{ \sc J. Han},
\newblock {Local Lipschitz regularity for functions satisfying a time-dependent dynamic programming principle.}
\newblock {ArXiv:1812.00646, 2018}.

\bibitem{Hoegl19}
{\sc F.A. H\o eg, P. Lindqvist}, Regularity of solutions of the parabolic normalized p-Laplace equation. \emph{Advances in Nonlinear Analysis}, 9(1):7--15, 2019.
\bibitem{IJS17}
{\sc C. Imbert, T. Jin,  L. Silvestre },
H\"older gradient estimates for a class of singular or degenerate parabolic equations.
{\em   Advances in Nonlinear Analysis} 8(1):845--867, 2019.

\bibitem{IS13}
{\sc C. Imbert, L. Silvestre},
$ C^{1,\alpha}$ regularity of solutions of some degenerate fully non-linear elliptic equations. {\em Adv. Math.} 233(1):196--206, 2013.




\bibitem{jinsl15}
{\sc T.~Jin,  L.~Silvestre},
\newblock H\"older gradient estimates for parabolic homogeneous $p$-{L}aplacian
equations.
\newblock {\em J. Math. Pures. Appl.} 108(1):63--87, 2017.




\bibitem{julm}
{\sc P. Juutinen, P. Lindqvist, J.J.  Manfredi},     
 \newblock  On the equivalence of viscosity solutions and weak solutions for a quasi-linear equation. {\em SIAM J. Math. Anal.} 33 (3):699--717, 2011. 

\bibitem{KM12} {\sc T. Kuusi,  G. Mingione}
\newblock New perturbation methods for nonlinear parabolic problems. 
{\em J. Math. Pures Appl}. 98(4):390--427, 2012.

\bibitem{LU68}
{\sc O.A. Ladyzhenskaya, V. A. Solonnikov,  N.N. Uraltseva}
\newblock  Linear and quasilinear  equations of parabolic type.
\newblock{\em Translated from the Russian by S. Smith. Translations of Mathematical Monographs, Vol. 23 American Mathematical Society, Providence, R.I. 1968 xi+648 pp.} 1968.



\bibitem{Lieb96}
{\sc G. M. Lieberman},
\newblock
Second order parabolic differential equations.
\newblock {\em World Scientific Publishing Co., Inc., River Edge, NJ, xii+439 pp}, 1996.



\bibitem{MPR10}
{ \sc J.J.~Manfredi, M.~Parviainen, J.D.~Rossi}, An asymptotic mean value characterization for a class of nonlinear parabolic equations related to tug-of-war games. \emph{SIAM J. Math. Anal.} 42(5):2058--2081, 2010.


\bibitem{OS}
{\sc M. Ohnuma, K.  Sato},
Singular degenerate parabolic equations with applications to the $p$-Laplace diffusion equation. 
{\em Comm. Partial Differential Equations.} 22(3-4):381--411, 1997. 

\bibitem{parviainenr16}
{ \sc M.~Parviainen, E.~Ruosteenoja},
\newblock Local regularity for time-dependent tug-of-war games with varying probabilities.
\newblock {\em J. Differential Equations}, 261(2):1357--1398, 2016.



\bibitem{PV18}
{\sc M. Parviainen, J.L. V\'azquez}
Equivalence between radial solutions of different parabolic gradient-diffusion equations and applications. To appear in {\em Ann. Sc. Norm. Super. Pisa Cl. Sci.}  Preprint: \url{https://arxiv.org/abs/1801.00613}, 2018.


\bibitem{savin}
{\sc O. Savin} Small perturbation solutions for elliptic equations. {\em Comm. Partial Differential Equations} 32(4-6):557--578, 2007.

\bibitem{wang94}
{ \sc L. Wang}, Compactness methods for certain degenerate elliptic equations. {\em J. Differential Equations}
107(2):341--350,1994.

\bibitem{wangyu}
{ \sc Y. Wang}, 
Small perturbation solutions for parabolic equations. {\em
Indiana Univ. Math. J.} 62(2):671--697, 2013.
 

\bibitem{wie86} {\sc M. Wiegner}. On $C^\alpha$-regularity of the gradient of solutions of degenerate parabolic systems.
{\em Ann. Mat. Pura Appl.} 145(4):385--405, 1986.


\end{thebibliography}
\end{document}